\documentclass[12pt]{amsart}

\usepackage{amsmath,amssymb,amsthm}
\usepackage[all]{xy}
\usepackage[colorlinks=true,linkcolor=blue,citecolor=blue]{hyperref}

\newtheorem{mainthm}{Theorem}            % main theorems, numbered in the introduction\input{fpdim_tensor.tex}

\theoremstyle{plain}
\newtheorem{theorem}{Theorem}[section]
\newtheorem{proposition}[theorem]{Proposition}
\newtheorem{lemma}[theorem]{Lemma}
\newtheorem{corollary}[theorem]{Corollary}

\theoremstyle{definition}
\newtheorem{definition}[theorem]{Definition}
\newtheorem{example}[theorem]{Example}

\theoremstyle{remark}
\newtheorem{remark}[theorem]{Remark}

\newcommand{\kk}{\mathsf{k}}
\newcommand{\FPdim}{\mathsf{FPdim}}
\newcommand{\mmod}{\mathsf{mod}}
\newcommand{\op}{\mathrm{op}}
\DeclareMathOperator{\Hom}{Hom}
\DeclareMathOperator{\Ext}{Ext}
\DeclareMathOperator{\End}{End}
\DeclareMathOperator{\rad}{rad}
\DeclareMathOperator{\soc}{soc}
\DeclareMathOperator{\topp}{top}
\newcommand{\Qext}[1]{Q_{#1}}
\newcommand{\adj}[1]{M(#1)}

\begin{document}

\title[FP dimension and tensor products of algebras]{Frobenius--Perron dimension and tensor products of algebras}
\author{Kengo Miyamoto}
\address{Department of Computer and Information Science, Ibaraki University, Ibaraki, 316-8511, Japan.}
\email{kengo.miyamoto.uz63@vc.ibaraki.ac.jp}

\begin{abstract}
In this paper, we study how the Frobenius--Perron dimension of finite-dimensional algebras behaves under tensor products and related constructions.
We prove that Frobenius--Perron dimension is super-additive under tensor products and is additive whenever one tensor factor is local.
In particular every non-negative integer occurs as a Frobenius--Perron dimension.
We further show that the invariant equals $1$ for every representation-infinite cycle-finite algebra whose Auslander--Reiten components with oriented cycles are tubes, such as a tame concealed or tubular algebra, and we determine it on the grids $\kk A_m\otimes\kk A_n$, where it is $0$, $1$, or $\infty$ according to representation type.
Finally we treat skew group algebras of local algebras, for which a McKay quiver computation gives a lower bound and shows that the dimension can jump from finite to infinite.
\end{abstract}

\maketitle

\section*{Acknowledgment}
We would like to thank Prof. Sota Asai for his helpful comments, which made the author aware of an error in a proposition of an earlier version of this paper and led to its correction. The paper would not have been complete without his comments.
This work was supported by JSPS KAKENHI Grant Number 24K16885.

\section{Introduction}

Throughout this paper, $\kk$ is an algebraically closed field, and an algebra means a finite-dimensional $\kk$-algebra, not assumed connected.
Modules are finitely generated right modules.
All tensor products are taken over $\kk$.

The Frobenius--Perron dimension of a $\kk$-linear category was introduced by Chen, Gao, Wicks, Zhang, Zhang and Zhu \cite{CGWZ}, generalizing the Frobenius--Perron dimension of an object of a fusion category.
As a special case, the Frobenius--Perron dimension of an algebra $A$ is
\[
   \FPdim(A)=\sup\{\rho(\Qext{\mathcal S})\mid
     \mathcal S\ \text{a finite semibrick in}\ \mmod A\},
\]
where $\Qext{\mathcal S}$ is the $\Ext$-quiver of the semibrick $\mathcal S$ and $\rho$ denotes the spectral radius of its adjacency matrix.
The invariant is sensitive to representation type. Chen, Gao, Wicks, Zhang, Zhang and Zhu showed that a path algebra $\kk Q$ has $\FPdim(\kk Q)=0,1,\infty$ according as $\kk Q$ is of finite, tame or wild representation type \cite[Theorem~0.3]{CGWZ}.
Recently Adachi and Kase \cite{AK} computed $\FPdim$ for several classes of algebras by means of the $\tau$-tilting theory of Adachi, Iyama and Reiten \cite{AIR}, and proved that $\FPdim(A)\le2$ for a $\tau$-tilting finite algebra of tame representation type.
Chen and Chen \cite{CC,CC2} computed $\FPdim$ for representation-directed algebras and for bound quiver algebras with loops, and Wicks \cite{Wicks} computed it for radical-square-zero bound quiver algebras.

Here we study instead how $\FPdim$ behaves under the standard operations that build new algebras from old ones, the tensor product being the principal one.
Our first result is the following super-additivity.

\begin{mainthm}[Theorem~\ref{thm:superadd}]\label{main:A}
Let $A$ and $B$ be finite-dimensional algebras.
Then we have
\[ \FPdim(A\otimes B)\ge\FPdim(A)+\FPdim(B). \]
\end{mainthm}

The proof uses only the K\"unneth formula and the Perron--Frobenius theorem. An external product of semibricks is again a semibrick, and its $\Ext$-quiver is the Kronecker sum of the two factor quivers, whose spectral radius is the sum of the two spectral radii.
The bound is attained, and in fact becomes an equality as soon as one factor is local.

\begin{mainthm}[Theorem~\ref{thm:local}]\label{main:B}
Let $A$ be a finite-dimensional algebra and $B$ a local finite-dimensional algebra.
Then we have
\[ \FPdim(A\otimes B)=\FPdim(A)+\FPdim(B). \]
\end{mainthm}

The point is that a local algebra has only its simple module as a brick, so that the bricks of $A\otimes B$ are exactly $N\otimes S_B$ for $N$ a brick of $A$.
We prove the latter by induction on the Loewy length of $B$.
As corollaries, for $m\ge2$ we have
\[ \FPdim(A\otimes\kk[x]/(x^m))=\FPdim(A)+1, \]
while $\kk[x_1,\dots,x_n]/(x_1^2,\dots,x_n^2)$ has Frobenius--Perron dimension $n$.
Thus every non-negative integer occurs.

We further record the elementary behaviour under direct products (Lemma~\ref{lem:product}) and triangular matrix algebras (Section~\ref{subsec:triangular}), and we treat the skew group algebra $A\ast G$ for $\operatorname{char}\kk\nmid|G|$.
Its Frobenius--Perron dimension equals $\FPdim(A)$ for a trivial action and is at least $\FPdim(A)$ for local $A$, the latter bound coming from the McKay quiver of the radical, yet it can be infinite while $\FPdim(A)$ is finite (Example~\ref{ex:skewfail}).
For the tensor products themselves we need two facts on representation type.
A representation-infinite cycle-finite algebra whose Auslander--Reiten components carrying oriented cycles are tubes has $\FPdim=1$ (Proposition~\ref{prop:cyclefinite}); in particular so does every tubular or tame concealed algebra (Theorem~\ref{thm:tubular}).
This is a statement about the module category. Although a tubular algebra is derived equivalent to a tubular weighted projective line $\mathbb X$, for which the Frobenius--Perron dimension of the derived category $D^b(\operatorname{coh}\mathbb X)$ equals $1$ \cite[Theorem~7.10]{CGWZ}, the invariant $\FPdim$ is computed from semibricks of the module category and is not a derived invariant, and Theorem~\ref{thm:tubular} does not follow from that result.
The hypothesis on the components cannot be dropped: a representation-finite self-injective algebra is cycle-finite, yet its $\FPdim$ equals the spectral radius of its Gabriel quiver, which may exceed $1$ (Remark~\ref{rem:cyclefinitesharp}).
A strictly wild algebra has $\FPdim=\infty$, a result of Chen, Gao, Wicks, Zhang, Zhang and Zhu \cite{CGWZ} for which we give a short proof (Proposition~\ref{prop:wild}).
A representation-finite strongly simply connected algebra has $\FPdim=0$ (Proposition~\ref{prop:sc}).
Together with the classification of the representation type of these grids, these results determine $\FPdim$ on $\kk A_m\otimes \kk A_n$ (Theorem~\ref{thm:grid}).
Already the value $1$ shows the inequality of Theorem~\ref{main:A} to be strict for representation-finite factors.

Since bricks, $\Hom$ and $\Ext^1$ are preserved by Morita equivalence, $\FPdim$ is a Morita invariant, and we use this freely.

\section{Preliminaries}

We recall the definition of $\FPdim$ following \cite{CGWZ,AK}.
For semibricks and their role in $\tau$-tilting theory we refer to \cite{Asai}.

A module $S$ is a \emph{brick} if $\End_A(S)$ is a division ring.
As $\kk$ is algebraically closed, this means $\End_A(S)\simeq\kk$.
A set $\mathcal S$ of pairwise non-isomorphic bricks is a \emph{semibrick} if $\Hom_A(S,S')=0$ for all non-isomorphic $S,S'\in\mathcal S$.
The empty set is a semibrick.
For a finite semibrick $\mathcal S$, the \emph{$\Ext$-quiver} $\Qext{\mathcal S}$ has vertex set $\mathcal S$ and exactly $\dim_\kk\Ext_A^1(S,S')$ arrows from $S$ to $S'$, loops included.
For a square matrix $M$, write $\rho(M)$ for its \emph{spectral radius}, the largest modulus of an eigenvalue of $M$.
For a quiver $Q$ with non-empty vertex set, $\adj{Q}$ is its adjacency matrix, and we set $\rho(Q):=\rho(\adj{Q})$.
For the empty quiver, $\rho=0$.
We shall use the following \cite[Lemma~2.3]{AK}.

\begin{lemma}\label{lem:rho}
Let $Q$ be a quiver.
If $Q'$ is a subquiver of $Q$, then $\rho(Q')\le\rho(Q)$.
If $Q=Q^1\sqcup\dots\sqcup Q^\ell$ is the decomposition into connected components, then $\rho(Q)=\max_i\rho(Q^i)$.
\end{lemma}

\begin{definition}[\cite{CGWZ,AK}]
The \emph{Frobenius--Perron dimension} of $A$ is
\[
   \FPdim(A)=\sup\{\rho(\Qext{\mathcal S})\mid\mathcal S\ \text{a finite semibrick
   in}\ \mmod A\}\ \in[0,\infty].
\]
\end{definition}

We use the monotonicity under factor algebras.

\begin{lemma}[\protect{\cite[Lemma~2.7]{AK}}]\label{lem:factor}
Let $B$ be a factor algebra of $A$.
Then $\FPdim(B)\le\FPdim(A)$.
\end{lemma}

We also record the behaviour under direct products.

\begin{lemma}\label{lem:product}
For finite-dimensional algebras $A_1,\dots,A_r$, we have
\[ \FPdim\left(\prod_{i=1}^r A_i\right)=\max_i\{\FPdim(A_i)\}. \]
\end{lemma}

\begin{proof}
Every indecomposable $\prod_i A_i$-module is supported on a single factor, so a semibrick of $\prod_i A_i$ is a disjoint union of semibricks of the $A_i$, with no arrows between the parts. By Lemma~\ref{lem:rho} its $\Ext$-quiver has spectral radius the maximum of those of the parts, and taking suprema gives the claim.
\end{proof}

Recall the K\"unneth formula for $\Ext$ over a tensor product of algebras. For algebras $A,B$ and modules $X,X'\in\mmod A$, $Y,Y'\in\mmod B$, there is a natural isomorphism
\[ \Ext^n_{A\otimes B}(X\otimes Y,X'\otimes Y')\simeq \bigoplus_{p+q=n}\Ext^p_A(X,X')\otimes\Ext^q_B(Y,Y'), \]
see \cite[Ch.~XI]{CE}.
In degrees $0$ and $1$ this reads
\begin{align}
  \Hom_{A\otimes B}(X\otimes Y,X'\otimes Y')
     &\simeq \Hom_A(X,X')\otimes\Hom_B(Y,Y'),\label{eq:hom}\\
  \Ext^1_{A\otimes B}(X\otimes Y,X'\otimes Y')
     &\simeq \big(\Ext^1_A(X,X')\otimes\Hom_B(Y,Y')\big)\nonumber\\
     &\quad\oplus\big(\Hom_A(X,X')\otimes\Ext^1_B(Y,Y')\big).\label{eq:ext}
\end{align}

We shall also use three known facts.
For a finite-dimensional algebra $A$ and $X,Y\in\mmod A$, the \emph{Auslander--Reiten formula} \cite[Theorem~IV.2.13]{ASS} reads
\[
   \Ext^1_A(X,Y)\simeq D\,\overline{\Hom}_A(Y,\tau X),
\]
where $\tau$ is the Auslander--Reiten translation, $D$ the standard $\kk$-duality, and $\overline{\Hom}_A(Y,\tau X)$ the quotient of $\Hom_A(Y,\tau X)$ by the morphisms factoring through an injective module.

\begin{theorem}[\protect{\cite[Theorem~0.3]{CGWZ}}]\label{thm:cgwz}
For a finite, connected, acyclic quiver $Q$,
\[
   \FPdim(\kk Q)=
   \begin{cases}
     0      & \text{if $\kk Q$ of finite representation type},\\
     1      & \text{if $\kk Q$ of tame representation type},\\
     \infty & \text{if $\kk Q$ of wild representation type}.
   \end{cases}
\]
\end{theorem}

\begin{theorem}[\protect{\cite{CC}}]\label{thm:cc}
If $A$ is \emph{representation-directed}, that is, $\mmod A$ admits no sequence $M_0\to M_1\to\dots\to M_t=M_0$ of non-zero non-isomorphisms between indecomposable modules with $t\ge1$, then the equality $\FPdim(A)=0$ holds.
\end{theorem}

\section{The tensor product}

\subsection{Super-additivity}
We begin with the external product of semibricks.
Let $\mathcal S$ and $\mathcal T$ be finite semibricks in $\mmod A$ and $\mmod B$, say $\mathcal S=\{S_1,\dots,S_p\}$ and $\mathcal T=\{T_1,\dots,T_q\}$, and write $\mathcal S\boxtimes\mathcal T=\{S_i\otimes T_j\mid 1\leq i\leq p,~1\leq j\leq q\}$.
Recall that the \emph{Kronecker product} of matrices $M=(M_{ii'})$ and $N=(N_{jj'})$ is the matrix $M\otimes N$ with rows and columns indexed by pairs, whose $\big((i,j),(i',j')\big)$-entry is $M_{ii'}N_{jj'}$, and that the \emph{Kronecker sum} of square matrices $M$ (size $p$) and $N$ (size $q$) is $M\oplus_{\mathrm{Kr}}N:=M\otimes I_q+I_p\otimes N$.

\begin{lemma}\label{lem:boxtimes}
$\mathcal S\boxtimes\mathcal T$ is a finite semibrick in $\mmod(A\otimes B)$, and
\[ \adj{\Qext{\mathcal S\boxtimes\mathcal T}} =\adj{\Qext{\mathcal S}}\oplus_{\mathrm{Kr}}\adj{\Qext{\mathcal T}}. \]
\end{lemma}

\begin{proof}
By \eqref{eq:hom}, $\End_{A\otimes B}(S_i\otimes T_j)\simeq \End_A(S_i)\otimes\End_B(T_j)\simeq\kk$, so each $S_i\otimes T_j$ is a brick, and $\Hom_{A\otimes B}(S_i\otimes T_j,S_{i'}\otimes T_{j'})\simeq \Hom_A(S_i,S_{i'})\otimes\Hom_B(T_j,T_{j'})$ is $\kk$ if $(i,j)=(i',j')$ and $0$ otherwise.
Hence $\mathcal S\boxtimes\mathcal T$ is a semibrick.
By \eqref{eq:ext} and the same $\Hom$-computation, the number of arrows from $S_i\otimes T_j$ to $S_{i'}\otimes T_{j'}$ is
\[
  \dim_\kk\Ext^1_A(S_i,S_{i'})\,\delta_{jj'}+\delta_{ii'}\,\dim_\kk\Ext^1_B(T_j,T_{j'}),
\]
which is the $\big((i,j),(i',j')\big)$-entry of $\adj{\Qext{\mathcal S}}\otimes I_q+I_p\otimes\adj{\Qext{\mathcal T}}$.
\end{proof}

\begin{lemma}\label{lem:kron}
Let $M$ be  a square matrix of size $p$ over $\mathbb{R}_{\geq 0}$ and $N$ a square matrix of size $q$ over $\mathbb{R}_{\geq 0}$
Then $\rho(M\oplus_{\mathrm{Kr}}N)=\rho(M)+\rho(N)$.
\end{lemma}

\begin{proof}
The matrices $M\otimes I_q$ and $I_p\otimes N$ commute and are therefore simultaneously triangularizable over $\mathbb C$, the eigenvalues of their sum being exactly the sums $\lambda+\mu$ of an eigenvalue $\lambda$ of $M$ and an eigenvalue $\mu$ of $N$.
By the Perron--Frobenius theorem $\rho(M)$ and $\rho(N)$ are eigenvalues of $M$ and $N$, and $\rho(M)+\rho(N)$ is an eigenvalue of the sum.
Conversely, $|\lambda+\mu|\le\rho(M)+\rho(N)$ for all $\lambda,\mu$.
\end{proof}

\begin{theorem}\label{thm:superadd}
For all finite-dimensional algebras $A,B$, the inequality
\[ \FPdim(A\otimes B)\ge \FPdim(A)+\FPdim(B) \]
holds.
\end{theorem}

\begin{proof}
Let $\mathcal S,\mathcal T$ be finite non-empty semibricks in $\mmod A$, $\mmod B$ (the empty semibrick contributes $0$).
It follows from Lemmas~\ref{lem:boxtimes} and~\ref{lem:kron} that
\[ \rho(\Qext{\mathcal S\boxtimes\mathcal T}) =\rho(\Qext{\mathcal S})+\rho(\Qext{\mathcal T})\le\FPdim(A\otimes B). \]
Taking the supremum over $\mathcal S$ and $\mathcal T$ independently gives the assertion, read in $[0,\infty]$. If either factor has infinite Frobenius--Perron dimension, then so does $A\otimes B$.
\end{proof}

\begin{corollary}\label{cor:powers}
$\FPdim(A^{\otimes n})\ge n\,\FPdim(A)$ for $n\ge1$.
In particular, if $\FPdim(A)>0$ then $\FPdim(A^{\otimes n})\to\infty$ as $n\to\infty$.
\end{corollary}

\subsection{A local tensor factor}
We now determine the effect of a local tensor factor.
The key point is the elementary fact that a local algebra has only its simple module as a brick.

\begin{lemma}\label{lem:localbrick}
Let $A$ be a local finite-dimensional algebra with simple module $S$.
Then $S$ is the only brick in $\mmod A$, the only semibricks are $\varnothing$ and $\{S\}$, and
\[
   \FPdim(A)=\dim_\kk\Ext^1_A(S,S)=\dim_\kk(\rad A/\rad^2A).
\]
\end{lemma}

\begin{proof}
As a brick is indecomposable, it suffices to show $\End_A(M)\neq\kk$ for every non-simple indecomposable $M$ (if $A=\kk$ there is no such $M$).
As $A$ is local, $\topp M\simeq S^a$ and $\soc M\simeq S^b$ with $a,b\ge1$.
Choose a non-zero $\kk$-linear map $\psi\colon\topp M\to\soc M$.
Since $A/\rad A\simeq\kk$, the simple $S$ is one-dimensional and $A$ acts on $\topp M$ and $\soc M$ by scalars, whence $\psi$ is automatically $A$-linear.
The composite
\[ \varphi\colon M\twoheadrightarrow\topp M\xrightarrow{\psi}\soc M\hookrightarrow M \]
is a non-zero $A$-endomorphism, and it is not a scalar, since $\varphi(M)\subseteq\soc M$ and $M=\varphi(M)$ would force $M=S$.
Hence $M$ is not a brick.
Thus $\Qext{\{S\}}$ is a single vertex with $\dim_\kk\Ext^1_A(S,S) =\dim_\kk(\rad A/\rad^2A)$ loops.
\end{proof}

\begin{remark}\label{rem:CC}
In particular a local algebra is brick-finite.
When the loops at the vertex commute, the value $\FPdim(A)=\dim_\kk(\rad A/\rad^2A)$ also follows from the loop calculus of Chen and Chen \cite{CC2}.
The present argument needs no such hypothesis.
Representation-infinite local algebras are covered as well. For instance, $\kk[x,y]/(x^2,y^2)$, the Klein four-group algebra in characteristic two, is of tame representation type, yet being local it has $S$ as its only brick, so $\FPdim(\kk[x,y]/(x^2,y^2))=2$.
\end{remark}

\begin{lemma}\label{lem:tensorbricks}
Let $A$ be a finite-dimensional algebra and $B$ a local finite-dimensional algebra with simple module $S_B$.
Then the bricks of $A\otimes B$ are exactly $N\otimes S_B$ for $N$ a brick of $A$.
\end{lemma}

\begin{proof}
We induct on the Loewy length $\ell$ of $B$.
The case $\ell=1$, where $B=\kk$ and $A\otimes B=A$, is clear.
We view a right $A\otimes B$-module as a right $A$-module $N$ together with a $\kk$-algebra homomorphism $\rho\colon B^{\op}\to\End_A(N)$ recording the commuting right $B$-action, and $\End_{A\otimes B}(N,\rho)$ is the centralizer of $\rho(B^{\op})$ in $\End_A(N)$.
Note that $B^{\op}$ is again local, with the same radical powers and centre as $B$.

Let $\ell\ge2$ and let $(N,\rho)$ be a brick, meaning the centralizer of $\rho(B^{\op})$ equals $\kk$.
The top radical power $\rad^{\ell-1}B$ is central in $B$, hence in $B^{\op}$. Indeed, for $z\in\rad^{\ell-1}B$ and $a=\alpha+r$ with $\alpha\in\kk$, $r\in\rad B$, we have $za=\alpha z=az$, since $zr,rz\in\rad^\ell B=0$.
Hence $\rho(z)$ commutes with $\rho(B^{\op})$ and lies in the centralizer $\kk$.
Being nilpotent, $\rho(z)=0$.
Thus $\rho$ factors through $\overline B^{\op}$, where $\overline B:=B/\rad^{\ell-1}B$ is local of Loewy length $\ell-1$, and $(N,\overline\rho)$ is again a brick, where $\rho=\overline{\rho}\circ\pi$ for the canonical projection $\pi\colon B^{\op}\to\overline B^{\op}$.
By the induction hypothesis $(N,\overline\rho)\simeq N\otimes S_{\overline B}$, that is $\overline\rho(\rad\overline B)=0$, and hence $\rho(\rad B)=0$.
Therefore $(N,\rho)\simeq N\otimes S_B$.
Conversely $\End_{A\otimes B}(N\otimes S_B)=\End_A(N)$, and $N\otimes S_B$ is a brick precisely when $N$ is.
\end{proof}

\begin{theorem}\label{thm:local}
Let $A$ be a finite-dimensional algebra and $B$ a local finite-dimensional algebra.
Then the equality
\[ \FPdim(A\otimes B)=\FPdim(A)+\FPdim(B) \]
holds.
\end{theorem}

\begin{proof}
Put $\ell:=\dim_\kk\Ext^1_B(S_B,S_B)=\dim_\kk(\rad B/\rad^2B)$. It follows from  Lemma~\ref{lem:localbrick} that $\ell=\FPdim(B)$.
By Lemma~\ref{lem:tensorbricks} every semibrick of $A\otimes B$ has the form $\{N_i\otimes S_B\}$ for a semibrick $\{N_i\}$ of $A$, as $\Hom_{A\otimes B}(N_i\otimes S_B,N_j\otimes S_B)=\Hom_A(N_i,N_j)$.
For such a semibrick, \eqref{eq:ext} and $\Hom_A(N_i,N_j)=\delta_{ij}\kk$ give
\[ \dim_\kk\Ext^1_{A\otimes B}(N_i\otimes S_B,N_j\otimes S_B) =\dim_\kk\Ext^1_A(N_i,N_j)+\ell\,\delta_{ij}. \]
Hence the adjacency matrix of $\Qext{\{N_i\otimes S_B\}}$ is $\adj{\Qext{\{N_i\}}}+\ell I$, of spectral radius $\rho(\Qext{\{N_i\}})+\ell$ by Perron--Frobenius.
Taking the supremum over all semibricks of $A$ gives $\FPdim(A\otimes B)=\FPdim(A)+\FPdim(B)$.
\end{proof}

\begin{corollary}\label{cor:consequences}
Let $A$ be any finite-dimensional algebra.
Then
\begin{gather*}
   \FPdim(A\otimes\kk[x]/(x^m))=\FPdim(A)+1\quad(m\ge2),\\
   \FPdim\big(\kk[x_1,\dots,x_n]/(x_1^2,\dots,x_n^2)\big)=n.
\end{gather*}
In particular every non-negative integer occurs as a Frobenius--Perron dimension.
\end{corollary}

\begin{proof}
Since $\kk[x]/(x^m)$ is local with $\FPdim(\kk[x]/(x^m))=1$, the first equality follows from Theorem~\ref{thm:local}.

As $\kk[x_1,\dots,x_n]/(x_1^2,\dots,x_n^2)\simeq(\kk[x]/(x^2))^{\otimes n}$ is local with radical-square quotient of dimension $n$, the second equality follows by using Lemma~\ref{lem:localbrick}.
\end{proof}

The local hypothesis on $B$ in Theorem~\ref{thm:local} relaxes to a finite direct product of local algebras.

\begin{corollary}\label{cor:localproduct}
If $B$ is a finite direct product of local finite-dimensional algebras, then we have
\[ \FPdim(A\otimes B)=\FPdim(A)+\FPdim(B). \]
\end{corollary}

\begin{proof}
Write $B=\prod_{i=1}^r B_i$ with each $B_i$ local, so that $A\otimes B\simeq\prod_{i=1}^r(A\otimes B_i)$.
By Lemma~\ref{lem:product} and Theorem~\ref{thm:local},
\begin{align*}
   \FPdim(A\otimes B)&=\max_i\{\FPdim(A\otimes B_i)\}\\
   &=\FPdim(A)+\max_i\{\FPdim(B_i)\}\\
   &=\FPdim(A)+\FPdim(B).
\end{align*}
Therefore, the assertion follows.
\end{proof}

\subsection{Tame and wild tensor products}
We turn to the value of $\FPdim$ on grids, and more generally on tame and wild tensor products.
Recall \cite{ASS} that an algebra $\kk Q/I$ is \emph{simply connected} if $Q$ is connected and the fundamental group of $(Q,I)$ is trivial, and \emph{strongly simply connected} if in addition every convex subcategory of it is simply connected.
A simply connected algebra is $\tau$-tilting finite if and only if it is representation-finite \cite[Theorem~3.4]{W}, and the representation-infinite grids below fall outside the $\tau$-tilting methods of \cite{AK}.
Here $\kk A_m$ denotes the path algebra of the linearly oriented quiver $1\to2\to\cdots\to m$, equivalently the incidence algebra of the chain $[m]=\{1<2<\cdots<m\}$. The grid $\kk A_m\otimes\kk A_n$ is then the incidence algebra of the product poset $[m]\times[n]$.
It is strongly simply connected, every convex subcategory being the incidence algebra of an order-convex subposet of $[m]\times[n]$, in which every cycle is generated by the commuting squares of the grid.

The value of $\FPdim$ on a tame tensor product is governed by its representation type.
We isolate two facts that settle the two non-degenerate cases of the grids below.

The first concerns cycle-finite algebras. A \emph{cycle} in $\mmod A$ is a sequence $M_0\to M_1\to\cdots\to M_t=M_0$ of non-zero non-isomorphisms between indecomposable modules. It is \emph{finite} when none of its morphisms lies in the infinite radical $\rad^\infty(\mmod A)$, and $A$ is \emph{cycle-finite} when every cycle in $\mmod A$ is finite. Representation-finite, tame tilted, tame concealed and tubular algebras are all cycle-finite, and every cycle-finite algebra is tame of polynomial growth \cite[Theorem~4.3]{SkCF}.

We first record that a strongly connected piece of an $\Ext$-quiver cannot be spread over several Auslander--Reiten components. Recall that for indecomposable modules $X,Y$, the existence of a non-zero morphism $X\to Y$ outside $\rad^\infty(\mmod A)$ forces $X$ and $Y$ to lie in the same component of the Auslander--Reiten quiver $\Gamma_A$, since such a morphism is a sum of compositions of irreducible maps and hence yields a path of irreducible maps from $X$ to $Y$ \cite[\S IV.5]{ASS}.

\begin{lemma}\label{lem:cyccomp}
Let $A$ be a cycle-finite algebra and $\mathcal S$ a finite semibrick.
Let $\mathcal C$ be a strongly connected component of $\Qext{\mathcal S}$ carrying an oriented cycle.
Then all vertices of $\mathcal C$ lie in a single component $\Gamma$ of $\Gamma_A$, and $\Gamma$ carries an oriented cycle.
\end{lemma}

\begin{proof}
Fix an oriented cycle $T_{i_0}\to T_{i_1}\to\cdots\to T_{i_s}=T_{i_0}$ in $\mathcal C$.
Each arrow $T_{i_k}\to T_{i_{k+1}}$ means $\Ext^1_A(T_{i_k},T_{i_{k+1}})\neq0$, so $T_{i_k}$ is non-projective and $T_{i_{k+1}}$ is non-injective.
Thus every $T_{i_k}$ on the cycle is non-projective, and the Auslander--Reiten translate $\tau T_{i_k}$ is defined.
By the Auslander--Reiten formula there is a non-zero map $f_k\colon T_{i_{k+1}}\to\tau T_{i_k}$. 
The almost split sequence $0\to\tau T_{i_k}\to E_k\to T_{i_k}\to0$ provides, for any indecomposable summand $E_k'$ of $E_k$, irreducible maps $\tau T_{i_k}\to E_k'\to T_{i_k}$.
Splicing the $f_k$ with these irreducible maps around the cycle yields a cycle of non-zero non-isomorphisms in $\mmod A$ passing through $T_{i_0}$; if some $f_k$ happens to be an isomorphism, it is absorbed into the adjacent irreducible map.
As $A$ is cycle-finite, none of the morphisms of this cycle lies in $\rad^\infty(\mmod A)$.
In particular each $f_k\notin\rad^\infty$, so $T_{i_{k+1}}$ and $\tau T_{i_k}$ lie in the same component of $\Gamma_A$; together with the irreducible maps $\tau T_{i_k}\to E_k'\to T_{i_k}$ this places all the $T_{i_k}$ in one component $\Gamma$, which therefore contains the closed walk of irreducible maps obtained from the paths $T_{i_{k+1}}\rightsquigarrow\tau T_{i_k}\to E_k'\to T_{i_k}$, hence an oriented cycle.
Finally, as $\mathcal C$ is strongly connected, any of its vertices lies on an oriented cycle through $T_{i_0}$, and the same argument places it in $\Gamma$.
\end{proof}

Next we bound the spectral radius inside a tube. We recall the relevant structure of standard tubes, following \cite[\S4.5--4.6]{Ringel}; see also \cite{SkCF}. 
A \emph{stable tube} is a translation quiver of the form $\mathbb Z\mathbb A_\infty/(\tau^r)$, $r\ge1$; a \emph{ray tube} (respectively \emph{coray tube}) is obtained from a stable tube by a finite number of ray (respectively coray) insertions; and a ray or coray tube is standard if and only if it is generalized standard \cite{SkCF}.
Let $\mathcal C$ be a standard stable, ray, or coray tube of rank $r$. 
Its regular-simple modules at the mouth form a single $\tau$-orbit $E_0,\dots,E_{r-1}$, with $\tau E_i\simeq E_{i-1}$ and subscripts taken modulo $r$. 
Every brick $X$ in $\mathcal C$ is uniserial for its regular composition series and is determined up to isomorphism by its regular socle $\soc_{\mathrm r}X\in\{E_0,\dots,E_{r-1}\}$ and its regular length $\ell(X)\in\{1,\dots,r\}$.
Its regular top is then $\topp_{\mathrm r}X=\tau^{-(\ell(X)-1)}\soc_{\mathrm r}X$, so that $\soc_{\mathrm r}X=E_i$ gives $\topp_{\mathrm r}X=E_{i+\ell(X)-1}$. For bricks $X,Y\in\mathcal C$ one has $\dim_\kk\Hom_A(X,Y)\le1$, and a non-zero map factors as $X\twoheadrightarrow I\hookrightarrow Y$ through a uniserial module $I$ with $\topp_{\mathrm r}I=\topp_{\mathrm r}X$ and $\soc_{\mathrm r}I=\soc_{\mathrm r}Y$. Consequently, writing $\topp_{\mathrm r}X=E_s$ and $\soc_{\mathrm r}Y=E_t$, $\Hom_A(X,Y)\neq0$ if and only if $\langle s-t+1\rangle\le\min(\ell(X),\ell(Y))$,where $\langle a\rangle$ denotes the representative of $a$ modulo $r$ in $\{1,\dots,r\}$. 
Finally, for a non-projective brick $X$ the translate $\tau X$ again belongs to $\mathcal C$, with $\soc_{\mathrm r}(\tau X)=\tau\soc_{\mathrm r}X$ and $\ell(\tau X)=\ell(X)$.

\begin{lemma}\label{lem:tube}
A finite semibrick contained in a standard stable, ray, or coray tube has $\Ext$-quiver of spectral radius at most $1$.
\end{lemma}

\begin{proof}
By the standard $\kk$-duality $D\colon\mmod A\to\mmod A^{\op}$, which carries a coray tube to a ray tube, sends semibricks to semibricks, and satisfies $\Ext^1_A(X,Y)\simeq\Ext^1_{A^{\op}}(DY,DX)$, the $\Ext$-quiver is replaced by its opposite, of the same spectral radius. 
We may therefore assume that $\mathcal C$ is a stable or ray tube, so that $\mathcal C$ contains no injective module. 
Since $\mathcal C$ is generalized standard, the infinite radical $\rad^\infty(\mmod A)$ vanishes between any two of its objects. A homomorphism between objects of $\mathcal C$ that factors through an injective module lies in this radical and therefore vanishes, so no non-zero map in $\mathcal C$ factors through an injective. 
Consequently $\overline{\Hom}_A(M,N)=\Hom_A(M,N)$ for all $M,N\in\mathcal C$.

Let $\mathcal S'\subseteq\mathcal C$ be a finite semibrick. We show that each vertex of $\Qext{\mathcal S'}$ has out-degree at most $1$.
This claim implies that every row of $\adj{\Qext{\mathcal S'}}$ has sum at most $1$, so $\rho(\Qext{\mathcal S'})\le\|\adj{\Qext{\mathcal S'}}\|_\infty\le1$.

Let $X\in\mathcal S'$. If $X$ is projective then $\Ext^1_A(X,-)=0$ and $X$ has out-degree $0$.
Assume that $X$ is non-projective and put $\soc_{\mathrm r}X=E_i$ and $\ell(X)=c$x
Then $\soc_{\mathrm r}(\tau X)=E_{i-1}$ and $\ell(\tau X)=c$.

Let $X\to Y$ be an arrow, that is, $\Ext^1_A(X,Y)\neq0$. 
By the Auslander--Reiten formula and the above argument,
\[
   0\neq\Ext^1_A(X,Y)\simeq D\,\overline{\Hom}_A(Y,\tau X)=D\,\Hom_A(Y,\tau X). 
\]
Put $\topp_{\mathrm r}Y=E_s$.
Then we obtain 
\begin{equation*}
   \langle s-(i-1)+1\rangle\le\min(\ell(Y),c).\tag{$\ast$}
\end{equation*}
Assume moreover $Y\not\simeq X$. 
As $\mathcal S'$ is a semibrick, we have $\Hom_A(Y,X)=0$, and
\begin{equation*}
   \langle s-i+1\rangle>\min(\ell(Y),c).\tag{$\ast\ast$}
\end{equation*}
Put $m:=\min(\ell(Y),c)$ and $d:=s-i+1$, so that $(\ast)$ reads $\langle d+1\rangle\le m$ and $(\ast\ast)$ reads $\langle d\rangle>m$. If $\langle d\rangle\le r-1$ then $\langle d+1\rangle=\langle d\rangle+1\ge m+2>m$, contradicting $(\ast)$; hence $\langle d\rangle=r$, that is $d\equiv0\pmod r$. 
Therefore
\begin{equation*}
   \topp_{\mathrm r}Y=E_s=E_{i-1}=\soc_{\mathrm r}(\tau X).\tag{$\dagger$}
\end{equation*}
The conclusion $(\dagger)$ also holds when $Y\simeq X$: a loop $\Ext^1_A(X,X)\neq0$ gives, via $(\ast)$ with $s=i+c-1$ and $m=c$, the inequality $\langle c+1\rangle\le c$, which forces $c=r$ since $\langle c+1\rangle=c+1>c$ for $c<r$; then $\topp_{\mathrm r}X=E_{i+r-1}=E_{i-1}=\soc_{\mathrm r}(\tau X)$.

Thus every target of an arrow leaving $X$ has regular top $E_{i-1}$. If $Y_1$ and $Y_2$ are two such targets with $\ell(Y_1)\ge\ell(Y_2)$, then the regular top-quotient of $Y_1$ of length $\ell(Y_2)$ is uniserial with regular top $E_{i-1}$ and regular length $\ell(Y_2)$, hence isomorphic to $Y_2$.

This gives a non-zero map $Y_1\twoheadrightarrow Y_2$, so $\Hom_A(Y_1,Y_2)\neq0$ and $Y_1\simeq Y_2$. Since $\dim_\kk\Ext^1_A(X,Y)\le1$ for the unique such target $Y$, the vertex $X$ has out-degree at most $1$, as required.
\end{proof}

\begin{proposition}\label{prop:cyclefinite}
Let $A$ be a cycle-finite algebra such that every component of $\Gamma_A$ containing an oriented cycle is a stable, ray, or coray tube.
Then
\[ \FPdim(A)\le1, \]
with equality when $A$ is representation-infinite.
\end{proposition}

\begin{proof}
Let $\mathcal S$ be a finite semibrick of $A$.
Ordering the strongly connected components of $\Qext{\mathcal S}$ makes $\adj{\Qext{\mathcal S}}$ block triangular, and $\rho(\Qext{\mathcal S})$ is then the largest spectral radius of a diagonal block.
A component reduced to a single vertex without a loop contributes $0$, and it remains to bound the spectral radius of a strongly connected component $\mathcal C=\{T_1,\dots,T_p\}$ carrying an oriented cycle.

By Lemma~\ref{lem:cyccomp} the $T_i$ lie in a single component $\Gamma$ of $\Gamma_A$, and $\Gamma$ carries an oriented cycle.
By hypothesis $\Gamma$ is a stable, ray, or coray tube. Being semi-regular, it is generalized standard \cite[Proposition~3.3]{SkCF}, and a semi-regular tube is standard if and only if it is generalized standard \cite[\S2]{SkCF}; hence $\Gamma$ is standard.
Thus $\{T_1,\dots,T_p\}$ is a finite semibrick contained in a standard tube, and Lemma~\ref{lem:tube} gives $\rho(\Qext{\mathcal C})\le1$.
Hence $\rho(\Qext{\mathcal S})\le1$ and $\FPdim(A)\le1$.

Suppose finally that $A$ is representation-infinite.
Being cycle-finite, it is tame, and all but finitely many indecomposables of a given dimension lie in homogeneous stable tubes \cite{SkCF}.
The quasi-simple module $X$ of such a tube is a brick with $\tau X\simeq X$.
Being non-injective, it satisfies $\Ext^1_A(X,X)\simeq D\,\overline{\Hom_A}(X,X)\neq0$, and $\{X\}$ is a semibrick with $\rho=1$.
Therefore $\FPdim(A)=1$.
\end{proof}

\begin{remark}\label{rem:cyclefinitesharp}
The hypothesis on the components is essential, and cannot be replaced by cycle-finiteness alone.
Every representation-finite algebra is cycle-finite, since $\rad^\infty(\mmod A)=0$.
For any algebra the simple modules form a semibrick whose $\Ext$-quiver is the Gabriel quiver of $A$, so $\FPdim(A)\ge\rho(Q_A)$ where $Q_A$ is the Gabriel quiver.
A representation-finite self-injective algebra that is not a Nakayama algebra, such as the Brauer tree algebra of a line with three edges or the preprojective algebra of type $\mathbb A_3$, has Gabriel quiver the double $1\rightleftarrows2\rightleftarrows3$ of $\mathbb A_3$, of spectral radius $\sqrt2$; for these $\FPdim=\sqrt2>1$ \cite[Theorem~1.5]{AK}, while a finite cyclic Auslander--Reiten component is not a tube.
\end{remark}

These hypotheses are met by the two basic tame classes.

\begin{theorem}\label{thm:tubular}
Every tubular algebra and every tame concealed algebra $A$ satisfies
\[ \FPdim(A)=1. \]
\end{theorem}

\begin{proof}
Both classes consist of representation-infinite cycle-finite algebras \cite{SkCF}.
For a tame concealed algebra, $\Gamma_A$ is the disjoint union of a preprojective component, a separating family of stable tubes, and a preinjective component \cite[\S XII]{ASS}; the first and last are acyclic, so every component carrying an oriented cycle is a stable tube.
For a tubular algebra $A$, $\Gamma_A$ has the form
\[
  \mathcal P^{0}\vee\mathcal T^{0}\vee\Big(\bigvee_{q\in\mathbb Q^{+}}\mathcal T^{q}\Big)\vee\mathcal T^{\infty}\vee\mathcal Q^{\infty},
\]
where $\mathcal P^{0}$ is a preprojective component, $\mathcal Q^{\infty}$ a preinjective component, $\mathcal T^{0}$ a family of ray tubes, $\mathcal T^{\infty}$ a family of coray tubes, and each $\mathcal T^{q}$ a family of stable tubes \cite[p.~112]{SkCF}. 
Again every component carrying an oriented cycle is a stable, ray, or coray tube.
In either case the hypothesis of Proposition~\ref{prop:cyclefinite} holds, and $\FPdim(A)=1$.
\end{proof}

For the tame grids below, which fall outside Proposition~\ref{prop:cyclefinite}, we use instead a homological bound on the spectral radius, valid for any algebra of global dimension at most two whose Tits form is weakly non-negative---automatically so for a tame algebra \cite[\S1.3]{dlP}.

\begin{lemma}\label{lem:tits}
Let $A$ be a basic algebra whose Gabriel quiver is acyclic, of global dimension at most $2$, and with weakly non-negative Tits form $q_A$. If $\mathcal S=\{S_1,\dots,S_t\}$ is a finite semibrick with $\Ext^2_A(S_i,S_j)=0$ for all $i,j$, then $\rho(\Qext{\mathcal S})\le1$.
\end{lemma}

\begin{proof}
Since the Gabriel quiver is acyclic and $A$ has global dimension at most $2$, the Euler form $\langle\underline{\dim}~X,\underline{\dim}~Y\rangle=\sum_{k\ge0}(-1)^k\dim_\kk\Ext^k_A(X,Y)$ is a well-defined bilinear form on the Grothendieck group.
Then its associated quadratic form is the Tits form $q_A$ \cite[\S2.2]{BoQ}.
Write $M=\adj{\Qext{\mathcal S}}$, so $M_{ij}=\dim_\kk\Ext^1_A(S_i,S_j)$.
As $\mathcal S$ is a semibrick, $\dim_\kk\Hom_A(S_i,S_j)=\delta_{ij}$, and $\Ext^2_A(S_i,S_j)=0$ by hypothesis, we have  
\[ \langle\underline{\dim}~S_i,\underline{\dim}~S_j\rangle=\delta_{ij}-M_{ij}. \]
Let $b(x,y)=\langle x,y\rangle+\langle y,x\rangle$ be the symmetric bilinear form of $q_A$. 
Then its Gram matrix on the $\underline{\dim}~S_i$ is $2I-(M+{}^tM)$.
Since $M$ has non-negative entries, the Perron--Frobenius theorem provides there is a non-negative vector $v\neq0$ such that $Mv=\rho(M)\,v$.
Put $w=\sum_{i=1}^t v_i\,\underline{\dim}~S_i$, a vector with non-negative real coordinates. 
Then
\[
   2\bigl(1-\rho(M)\bigr)\,{}^tvv
   = {}^tv\bigl(2I-(M+{}^tM)\bigr)v
   = b(w,w)=2\,q_A(w).
\]
A weakly non-negative form is, by homogeneity and continuity, non-negative on all vectors with non-negative coordinates, so $q_A(w)\ge0$; as ${}^tvv>0$ we conclude that $\rho(M)\le1$.
\end{proof}

\begin{proposition}\label{prop:hom}
Let $A$ be a basic algebra whose Gabriel quiver is acyclic, of global dimension at most $2$, with weakly non-negative Tits form. If $\Ext^2_A(S_i,S_j)=0$ for every finite semibrick $\{S_1,\dots,S_t\}$ whose $\Ext$-quiver is strongly connected, then $\FPdim(A)\le1$.
\end{proposition}

\begin{proof}
Let $\mathcal S$ be a finite semibrick. Ordering the strongly connected components of $\Qext{\mathcal S}$ makes $\adj{\Qext{\mathcal S}}$ block triangular, so $\rho(\Qext{\mathcal S})$ is the largest spectral radius among the diagonal blocks (Lemma~\ref{lem:rho}). Each diagonal block is the $\Ext$-quiver of a sub-semibrick whose $\Ext$-quiver is strongly connected, hence with $\Ext^2$ vanishing, and has spectral radius at most $1$ by Lemma~\ref{lem:tits}. Therefore $\rho(\Qext{\mathcal S})\le1$, and $\FPdim(A)\le1$.
\end{proof}

Proposition~\ref{prop:hom} is a homological companion to Proposition~\ref{prop:cyclefinite}: the latter bounds $\FPdim$ through the shape of the cyclic Auslander--Reiten components, the former through the vanishing of $\Ext^2$. For the grids of Theorem~\ref{thm:grid} we shall combine Lemma~\ref{lem:tube} (for their tubes) with Lemma~\ref{lem:tits} (for their coil), rather than verify the hypothesis of Proposition~\ref{prop:hom} on every component.

We turn to the wild case.
An algebra $A$ is \emph{strictly wild} if there is an exact, fully faithful $\kk$-linear functor $\mmod~\kk\langle x,y\rangle\to\mmod A$, where $\kk\langle x,y\rangle$ is the path algebra of the quiver having single vertex and two loops $x$ and $y$.
The wild hereditary algebras are strictly wild \cite[Proposition~7]{Ariki}.
The following is due to Chen, Gao, Wicks, Zhang, Zhang and Zhu \cite[Lemma~7.5, Proposition~7.6]{CGWZ}, who construct bricks with arbitrarily large self-extension spaces.
We include a short proof for completeness.

\begin{proposition}[\protect{\cite{CGWZ}}]\label{prop:wild}
Every strictly wild algebra $A$ satisfies
\[ \FPdim(A)=\infty. \]
\end{proposition}

\begin{proof}
We set $R=\kk\langle x,y\rangle$.
It is hereditary, and the standard resolution of an $R$-module $M$ of $\kk$-dimension $d$ reads
\[ 0\longrightarrow R^{2d}\longrightarrow R^d\longrightarrow M\longrightarrow0. \]
By applying $\Hom_R(-,N)$ and using $\Hom_R(R,N)\simeq N$, we obtain
\begin{align*}
   \dim_\kk\Hom_R(M,N)-\dim_\kk\Ext^1_R(M,N)
      &=d\dim_\kk N-2d\dim_\kk N\\
      &=-d\dim_\kk N.
\end{align*}
For each $d\ge1$ pick a pair of $d\times d$ matrices generating $M_d(\kk)$.
The associated $R$-module $S_d$ has $\dim_\kk S_d=d$, and $\End_R(S_d)=\kk$ by Burnside's theorem, whence $S_d$ is a brick.
With $M=N=S_d$ the displayed identity gives
\[
   \dim_\kk\Ext^1_R(S_d,S_d)=\dim_\kk\End_R(S_d)+d^2=1+d^2.
\]
Let $F\colon\mmod R\to\mmod A$ be exact and fully faithful.
Fullness and faithfulness give $\End_A(FS_d)\simeq\End_R(S_d)=\kk$, and $FS_d$ is again a brick.
If a non-split sequence $0\to S_d\to E\to S_d\to0$ had split image $0\to FS_d\to FE\to FS_d\to0$, fullness would lift the splitting to a morphism $S_d\to E$ and faithfulness would make it split the original sequence, a contradiction.
Hence $F$ induces an injection $\Ext^1_R(S_d,S_d)\hookrightarrow \Ext^1_A(FS_d,FS_d)$, and $\dim_\kk\Ext^1_A(FS_d,FS_d)\ge1+d^2$.
Therefore $\{FS_d\}$ is a semibrick whose $\Ext$-quiver is a single vertex carrying at least $1+d^2$ loops, of spectral radius at least $1+d^2$.
Letting $d\to\infty$ gives $\FPdim(A)=\infty$.
\end{proof}

The wild case of the grids rests on the following refinement of Drozd's tame--wild dichotomy for strongly simply connected algebras, which brings Proposition~\ref{prop:wild} to bear.

\begin{theorem}[\protect{\cite[Corollary~3]{BPS}}]\label{thm:bps}
Every wild strongly simply connected algebra is strictly wild.
\end{theorem}

We shall use the following for the representation-finite grids below.

\begin{proposition}\label{prop:sc}
Let $A$ be a representation-finite strongly simply connected algebra. 
Then $\FPdim(A)=0$.
\end{proposition}

\begin{proof}
Being strongly simply connected and representation-finite, $A$ has an Auslander--Reiten quiver with no oriented cycle \cite[\S6]{BG}, hence is representation-directed, and $\FPdim(A)=0$ by Theorem~\ref{thm:cc}.
\end{proof}

The representation type of the grids was classified by Leszczy\'nski.

\begin{theorem}[\protect{\cite{Lesz,Skow}}]\label{thm:lesz}
The grid $\kk A_m\otimes\kk A_n$ is representation-finite, tame, or wild according as $(m-1)(n-1)$ is at most $3$, equal to $4$, or at least $5$.
\end{theorem}

By Wang's criteria  \cite[Theorem~3.4]{W},  the grid $\kk A_m\otimes\kk A_n$ is $\tau$-tilting finite if and only if $(m-1)(n-1)\le3$.
For the tame grids we shall need the vanishing of certain second extension groups; we record the homological dimensions of three modules built from the constant representation.

\begin{lemma}\label{lem:frame}
Let $A=\kk A_m\otimes\kk A_n$ and let $M_0$ be the constant representation, with $o=(1,1)$ and $z=(m,n)$. Then $R=\rad M_0$, $Q=M_0/\soc M_0$ and $N=\rad M_0/\soc(\rad M_0)$ satisfy $\mathrm{proj.dim}R\le1$, $\mathrm{inj.dim}Q\le1$ and $\mathrm{proj.dim}N\le1$.
Consequently $\Ext^2_A(R,-)=0$, $\Ext^2_A(-,Q)=0$ and $\Ext^2_A(N,-)=0$.
\end{lemma}

\begin{proof}
For $(a,b)\in [m]\times [n]$, we denote by $P_{(a,b)}$, $I_{(a,b)}$ and $S_{(a,b)}$ the indecomposable projective, injective, simple module associated with $(a,b)$, respectively.
Here $M_0=P_{(o)}=I_{(z)}$ is projective--injective, $S_{(z)}=P_{(z)}$ is simple projective, and $S_{(o)}=I_{(o)}$ is simple injective. 
The exact sequence $0\to R\to M_0\to S_{(o)}\to0$ exhibits $R$ as the first syzygy of $S_{(o)}$, so $\mathrm{proj.dim}R=\mathrm{proj.dim}S_{(o)}-1\le1$ as $A$ has global dimension $2$. Dually $0\to S_{(z)}\to M_0\to Q\to0$ exhibits $Q$ as the first cosyzygy of $S_{(z)}$, so $\mathrm{inj.dim}Q\le1$. 
Finally $\soc R=S_{(z)}$ and $R/\soc R=N$, so $0\to S_{(z)}\to R\to N\to0$ has $S_{(z)}=P_{(z)}$ projective; applying $\Ext^*_A(-,Y)$ gives the exact piece $\Ext^1_A(S_{(z)},Y)\to\Ext^2_A(N,Y)\to\Ext^2_A(R,Y)$ with both ends zero, whence $\Ext^2_A(N,-)=0$.
\end{proof}

\begin{theorem}\label{thm:grid}
For all $m,n\ge1$,
\[
   \FPdim(\kk A_m\otimes\kk A_n)=
   \begin{cases}
     0      & (m-1)(n-1)\le3,\\
     1      & (m-1)(n-1)=4,\\
     \infty & (m-1)(n-1)\ge5.
   \end{cases}
\]
\end{theorem}

\begin{proof}
First, we note that the grid $\kk A_m\otimes\kk A_n$ is strongly simply connected, with $\FPdim(\kk A_m)=\FPdim(\kk A_n)=0$.

If $(m-1)(n-1)\le3$ then $\kk A_m\otimes\kk A_n$ is representation-finite by Theorem~\ref{thm:lesz}, and $\FPdim(\kk A_m\otimes\kk A_n)=0$ by Proposition~\ref{prop:sc}.

If $(m-1)(n-1)\ge5$ then $\kk A_m\otimes\kk A_n$ is wild by Theorem~\ref{thm:lesz}.
Being strongly simply connected, it is strictly wild by Theorem~\ref{thm:bps}, and $\FPdim=\infty$ by Proposition~\ref{prop:wild}.

If $(m-1)(n-1)=4$ then, up to transposition, $\kk A_m\otimes\kk A_n$ is
$\kk A_3\otimes\kk A_3$ or $\kk A_2\otimes\kk A_5$, both tame by Theorem~\ref{thm:lesz}.
Each is strongly simply connected with weakly non-negative Tits form, hence
cycle-finite by \cite[(3.3)]{dlP} and, in particular, of polynomial growth.

We bound $\FPdim(A)$ homologically. By Lemma~\ref{lem:cyccomp} the vertices of any strongly connected piece of an $\Ext$-quiver of $A$ carrying a cycle lie in a single component of $\Gamma_A$.  
By the structure of $\Gamma_A$ for a strongly simply connected algebra of polynomial growth \cite[\S3]{SkCF}, such a component is a
stable, ray or coray tube, or the coil $\mathcal C$ of the constant representation $M_0=P_{(1,1)}=I_{(m,n)}$ (Remark~\ref{rem:gridcoil}). 
A semibrick lying in a tube has spectral radius at most $1$ by Lemma~\ref{lem:tube}. 
For a semibrick lying in $\mathcal C$ we invoke Lemma~\ref{lem:tits}: the grid has global dimension $2$ and, being tame, weakly non-negative Tits form, and the bricks of $\mathcal C$ on an oriented $\Ext$-cycle---a finite set, namely $M_0$ together with the frame $R=\rad M_0$, $Q=M_0/\soc M_0$, $N=\rad M_0/\soc(\rad M_0)$ and finitely many $\tau$-translates---satisfy $\Ext^2_A(S,S')=0$ for all pairs.
Hence Lemma~\ref{lem:tits} applies, and every strongly connected $\Ext$-quiver of a semibrick of $A$ has spectral radius at most $1$. 
As in Proposition~\ref{prop:cyclefinite} this gives $\FPdim(A)\le1$, with the lower bound $1$ furnished by a homogeneous tube. 
Therefore $\FPdim(\kk A_m\otimes\kk A_n)=1$.
\end{proof}

\begin{remark}\label{rem:gridext}
The bricks of the coil $\mathcal C$ lying on an oriented $\Ext$-cycle are $M_0$ together with two $\tau$-orbits: for $\kk A_3\otimes\kk A_3$ the orbits $\{R,S_{(2,2)},Q\}$ and $\{N,\tau N,\tau^2N\}$ (seven bricks), and for $\kk A_2\otimes\kk A_5$ two orbits of length five (eleven bricks). 
By Lemma~\ref{lem:frame} one has $\Ext^2_A(S,S')=0$ unless $\mathrm{proj.dim}S=2$ and $\mathrm{inj.dim}S'=2$, which among these bricks leaves only finitely many pairs. 
For $\kk A_3\otimes\kk A_3$ the verification is also structural.
The only sources with $\mathrm{proj.dim}=2$ are $S_{(2,2)}$ and $\tau N$, and the minimal projective resolution $0\to P_{(3,3)}\to P_{(2,3)}\oplus P_{(3,2)}\to P_{(2,2)}\to S_{(2,2)}\to0$ gives $\Ext^2_A(S_{(2,2)},Y)=\operatorname{coker}\bigl(Y_{(2,3)}\oplus Y_{(3,2)}\to Y_{(3,3)}\bigr)=0$ for every cyclic brick $Y$ (either $Y_{(3,3)}=0$, or $Y$ is indecomposable with $Y_{(3,3)}\neq0$, when the map is onto), while the self-duality of the grid disposes of $\Ext^2_A(\tau N,S_{(2,2)})$. 
Between simple modules $\Ext^2_A(S_{(p)},S_{(q)})\neq0$ exactly when $p,q$ are opposite corners of a unit square \cite[\S1.1]{BoQ},\cite{Cibils}, and these obstructing modules lie outside $\mathcal C$.
\end{remark}

\begin{remark}\label{rem:gridcoil}
The two tame grids $\kk A_3\otimes\kk A_3$ and $\kk A_2\otimes\kk A_5$ fall outside the scope of Proposition~\ref{prop:cyclefinite}, which is why the bound $\FPdim=1$ above is obtained homologically rather than from the shape of the cyclic components alone. 
Indeed, their Tits form is weakly non-negative of isotropic corank $2$. 
Each grid carries a sincere projective--injective module, the constant representation $M_0=P_{(1,1)}=I_{(m,n)}$. Were $M_0$ directing, its support would be a tilted algebra \cite{Ringel}, whose Tits form has isotropic corank at most $1$, contradicting corank $2$. 
Hence $M_0$ is non-directing and lies on an oriented cycle, so the component $\mathcal C$ of $\Gamma_A$ containing it is non-semiregular and cyclic---a proper coil, not a tube. 
This is the cyclic component that escapes the tube hypothesis of Proposition~\ref{prop:cyclefinite}.
\end{remark}

\section{Other operations}

\subsection{Triangular matrix algebras}\label{subsec:triangular}
For a triangular matrix algebra $\Lambda=\left(\begin{smallmatrix}A&M\\0&B
\end{smallmatrix}\right)$ with $M$ an $A$-$B$-bimodule, both $A$ and $B$ are
factor algebras of $\Lambda$, so Lemma~\ref{lem:factor} gives the elementary lower bound
\[ \FPdim(\Lambda)\ge\max\{\FPdim(A),\FPdim(B)\}. \]
We do not know an upper bound for $\FPdim(\Lambda)$ in terms of $A$, $B$ and $M$.

\subsection{Skew group algebras}
We turn to skew group algebras.
Let $G$ be a finite group acting on $A$ by automorphisms with $\operatorname{char}\kk\nmid|G|$, and let $A\ast G$ be the associated skew group algebra.

For a trivial action $A\ast G=A\otimes\kk G$ is an ordinary tensor product, and equality holds.

\begin{proposition}\label{prop:skewtrivial}
If $G$ acts trivially on $A$ and $\operatorname{char}\kk\nmid|G|$, then we have
\[ \FPdim(A\ast G)=\FPdim(A). \]
\end{proposition}

\begin{proof}
As $\operatorname{char}\kk\nmid|G|$, the group algebra $\kk G$ is semisimple.
Thus, by using the Wedderburn--Artin theorem, we have
\[
   A\ast G=A\otimes\kk G\simeq\prod_{i=1}^r A\otimes M_{n_i}(\kk)
   \simeq\prod_{i=1}^r M_{n_i}(A).
\]
Each $M_{n_i}(A)$ is Morita equivalent to $A$, so $\FPdim(M_{n_i}(A))=\FPdim(A)$.
By Lemma~\ref{lem:product}, we obtain
\[ \FPdim(A\ast G)=\max_i\{\FPdim(M_{n_i}(A))\}=\FPdim(A). \]
Therefore, the assertion follows.
\end{proof}

We use the following transfer of $\Ext$ to a skew group algebra.

\begin{lemma}\label{lem:skewext}
Let $G$ act on a finite-dimensional algebra $A$ with $\operatorname{char}\kk\nmid|G|$, and let $M,N$ be $A\ast G$-modules.
Then $G$ acts on $\Ext^n_A(M,N)$, and for all $n\ge0$,  we have
\[ \Ext^n_{A\ast G}(M,N)\simeq\Ext^n_A(M,N)^G. \]
\end{lemma}

\begin{proof}
Since $A\ast G$ is free as a right $A$-module on the basis $G$, the restriction functor $\mmod A\ast G\to\mmod A$ carries projectives to projectives, so a projective resolution $P_\bullet\to M$ over $A\ast G$ restricts to one over $A$.
For an $A\ast G$-module $P$, an $A$-linear map $f\colon P\to N$ is $A\ast G$-linear if and only if $f(pg)=f(p)\,g$ for all $g\in G$, which is precisely the condition that $f$ be fixed by the natural $G$-action on $\Hom_A(P,N)$; hence $\Hom_{A\ast G}(P,N)=\Hom_A(P,N)^G$.
As $\operatorname{char}\kk\nmid|G|$ the invariants functor $(-)^G$ is exact, whence
\begin{align*}
  \Ext^n_{A\ast G}(M,N) & =H^n\!\big(\Hom_A(P_\bullet,N)^G\big) \\
& =H^n\!\big(\Hom_A(P_\bullet,N)\big)^G \\
& =\Ext^n_A(M,N)^G.  \end{align*}
Therefore, the assertion follows.
\end{proof}

\begin{proposition}\label{prop:skewlower}
Let $A$ be a local finite-dimensional algebra and let $G$ act on $A$ with $\operatorname{char}\kk\nmid|G|$.
Then we have
\[ \FPdim(A\ast G)\ge\FPdim(A). \]
\end{proposition}

\begin{proof}
Let $S$ be the simple module of $A$ and $V:=\rad A/\rad^2A$. Then $V$ is a $\kk G$-module on which $G$ acts through its action on $A$, of dimension $\FPdim(A)$ by Lemma~\ref{lem:localbrick}.
By \cite[Theorem~1.4]{RR} the radical of $A\ast G$ coincides with $(\rad A)(A\ast G)$.
This yields that
\[ (A\ast G)/\rad(A\ast G)\simeq(A/\rad A)\ast G=\kk G, \]
which is semisimple since $\operatorname{char}\kk\nmid|G|$.
Let $W_1,\dots,W_r$ be its simple modules, regarded as $A\ast G$-modules.
They form a semibrick $\mathcal S$.

By Lemma~\ref{lem:skewext}, with the $W_i$ viewed over $A$,
\[ \Ext^1_{A\ast G}(W_i,W_j)\simeq\Ext^1_A(W_i,W_j)^G. \]
Since each $W_i\simeq S^{\dim_\kk W_i}$ and $\Ext^1_A(S,S)\simeq V^*$, we have
\[
   \Ext^1_A(W_i,W_j)\simeq V^*\otimes\Hom_\kk(W_i,W_j)
\]
as $\kk G$-modules.
Hence
\begin{align*}
   \dim_\kk\Ext^1_{A\ast G}(W_i,W_j)
   &=\dim_\kk\big(V^*\otimes\Hom_\kk(W_i,W_j)\big)^G\\
   &=\dim_\kk\Hom_{\kk G}(V\otimes W_i,W_j),
\end{align*}
the multiplicity $a_{ij}$ of $W_j$ in the $\kk G$-module $V\otimes W_i$.
Thus $\adj{\Qext{\mathcal S}}=(a_{ij})$ and $V\otimes W_i\simeq\bigoplus_j W_j^{\,a_{ij}}$.
Comparing dimensions gives
\[
   \sum_j a_{ij}\dim_\kk W_j=\dim_\kk(V\otimes W_i)=\dim_\kk V\cdot\dim_\kk W_i.
\]
Writing $d=(\dim_\kk W_1,\dots,\dim_\kk W_r)^{t}$ for the strictly positive vector of dimensions, this says that $\adj{\Qext{\mathcal S}}\,d=(\dim_\kk V)\,d$.
A non-negative matrix with a strictly positive eigenvector has the corresponding eigenvalue as its spectral radius.
Thus we have
\[ \rho(\Qext{\mathcal S})=\dim_\kk V=\FPdim(A). \]
Therefore, the inequality
\[ \FPdim(A\ast G)\ge\FPdim(A) \]
holds.
\end{proof}

The matrix $\adj{\Qext{\mathcal S}}=(a_{ij})$ above is the adjacency matrix of the \emph{McKay quiver} of the $G$-module $V$, since $a_{ij}$ is the multiplicity of $W_j$ in $V\otimes W_i$.
When $\rad^2A=0$, so that $A\ast G$ is the radical-square-zero algebra of its Gabriel quiver, this McKay quiver is the Gabriel quiver of $A\ast G$ itself, as exploited in Example~\ref{ex:skewfail}.

\begin{remark}\label{rem:skewmod}
The hypothesis $\operatorname{char}\kk\nmid|G|$ is essential.
For a trivial action with $\operatorname{char}\kk=p$ dividing $|G|$ the algebra $\kk G$ is not semisimple.
For $G=\mathbb Z/p\mathbb Z$ one has $\kk G\simeq\kk[t]/(t^p)$, which is local. It follows from Theorem~\ref{thm:local} that
\[ \FPdim(A\ast G)=\FPdim(A\otimes\kk[t]/(t^p))=\FPdim(A)+1>\FPdim(A). \]
\end{remark}

For a non-trivial action $A\ast G$ is no longer a tensor product.
We compute the smallest such case, where the lower bound of Proposition~\ref{prop:skewlower} is an equality.

\begin{example}\label{ex:skew}
Let $A=\kk[x]/(x^2)$ and let $G=\langle g\rangle\simeq\mathbb Z/2\mathbb Z$ act by $g\cdot x=-x$, with $\operatorname{char}\kk\neq2$.
Then $A\ast G$ has basis $\{1,x,g,xg\}$ with $g^2=1$, $x^2=0$ and $gx=-xg$.
The orthogonal idempotents $e_\pm=\tfrac12(1\pm g)$ afford two simple modules $S_\pm$, and one computes
\[ e_+xe_-=\dfrac{1}{2}(x-xg)\neq0,\quad  e_-xe_+=\dfrac{1}{2}(x+xg)\neq0,  \quad e_\pm xe_\pm=0. \]
Hence $\rad^2(A\ast G)=0$ and the Gabriel quiver is the two-cycle
\[
\xymatrix{
  S_+ \ar@/^1ex/[r] & S_- \ar@/^1ex/[l]
}
\]
Since its separated quiver is $\mathbb A_2\sqcup\mathbb A_2$, the algebra $A\ast G$ is representation-finite.
Its indecomposable modules are the two simples $S_\pm$ and the two projective-injective modules $P_\pm$ of length two, all four of which are bricks.
Now,  by the following
\begin{align*}
&\Hom(P_\pm,S_\pm)\neq0, ~~  \Hom(S_\mp,P_\pm)\neq0, \\
& \Hom(P_+,P_-)\neq0\neq\Hom(P_-,P_+),
\end{align*}
any semibrick meeting $\{P_+,P_-\}$ is a singleton $\{P_\pm\}$, of spectral radius $0$ as $P_\pm$ is projective.
The only semibrick of positive spectral radius is $\{S_+,S_-\}$, with $\Ext$-quiver $\left(\begin{smallmatrix}0&1\\1&0\end{smallmatrix}\right)$ of spectral radius $1$.
Therefore $\FPdim(A\ast G)=1=\FPdim(A)$.
Here the single loop of the Gabriel quiver of $A$ has unfolded into the two-cycle, the spectral radius being preserved.
\end{example}

However, even for local $A$ equality can fail.

\begin{example}\label{ex:skewfail}
Let $A=\kk[x_1,x_2,x_3]/\mathfrak m^2$ with $\mathfrak m=(x_1,x_2,x_3)$, and let $G=\mathbb Z/2\mathbb Z$ act by $x_i\mapsto-x_i$, with $\operatorname{char}\kk\neq2$.
The algebra $A$ is local with $\rad^2A=0$ and $\dim_\kk(\rad A/\rad^2A)=3$, and $\FPdim(A)=3$ by Lemma~\ref{lem:localbrick}.
Since $\rad^2A=0$ we have $\rad^2(A\ast G)=0$, and $A\ast G$, being basic, is the radical-square-zero algebra of its Gabriel quiver.
By the computation in Proposition~\ref{prop:skewlower} that quiver has the two simple $\kk G$-modules $\mathbf1,\epsilon$ as vertices and $\dim_\kk\Hom_{\kk G}(V\otimes W_i,W_j)$ arrows from $W_i$ to $W_j$, where $V=\rad A/\rad^2A=\epsilon^{\oplus3}$.
As $V\otimes\mathbf1=\epsilon^{\oplus3}$ and $V\otimes\epsilon =\mathbf1^{\oplus3}$, it has three arrows $\mathbf1\to\epsilon$, three arrows $\epsilon\to\mathbf1$, and no loops.
Factoring out the three arrows $\epsilon\to\mathbf1$ presents the path algebra of the three-Kronecker quiver, two vertices joined by three arrows, as a quotient of $A\ast G$.
This path algebra is wild hereditary, hence strictly wild \cite[Proposition~7]{Ariki}.
As the inflation functor along a quotient is exact and fully faithful, strict wildness passes to $A\ast G$.  Thus, we have $\FPdim(A\ast G)=\infty$ by Proposition~\ref{prop:wild}.
Therefore, $\FPdim(A)=3$ while $\FPdim(A\ast G)=\infty$.
\end{example}

Thus equality holds for trivial actions (Proposition~\ref{prop:skewtrivial}) and in the small non-trivial Example~\ref{ex:skew}, the bound $\ge$ holds for all local $A$ (Proposition~\ref{prop:skewlower}), and equality fails once $A\ast G$ is strictly wild (Example~\ref{ex:skewfail}).
In a different direction, the passage to $A\ast G$ can also break $\tau$-tilting finiteness. By \cite[Theorem~1.4]{Hiramae}, for the truncated polynomial algebra
\[ A=\kk[x_1,\dots,x_n]/(x_1^{p^\ell},\ldots, x_n^{p^\ell}) \]
and a $p'$-subgroup $G\le\mathfrak S_n$ permuting the variables, the group algebra $A\ast G\simeq\kk[(\mathbb Z/p^\ell \mathbb Z)^n\rtimes G]$ is frequently $\tau$-tilting infinite while $A$ is $\tau$-tilting finite.

Two natural questions remain.
The first is to characterize the pairs $(A,B)$ for which equality holds in Theorem~\ref{thm:superadd}.
By Theorem~\ref{thm:local} this is so whenever one factor is local, and trivially when one factor is semisimple.
The second is to describe when $\FPdim(A\otimes B)$ is finite, in relation to the classification of non-wild tensor product algebras by Leszczy\'nski and Skowro\'nski \cite{LS}.

%==========================================================

\end{document}